\documentclass[11pt]{amsart}
\usepackage{amsmath,amsfonts,amssymb,mathrsfs}
\usepackage{colortbl}
\definecolor{black}{rgb}{0.0, 0.0, 0.0}
\definecolor{red}{rgb}{1.0, 0.5, 0.5}

\title[   ]{Global weak solutions for  Kolmogorov-Vicsek type equations with orientational interactions}

\author[Gamba]{Irene M. Gamba}
\address[Irene M. Gamba]{\newline Department of Mathematics and ICES, \newline The University of Texas at Austin, Austin, TX 78712, USA}
\email{gamba@math.utexas.edu}

\author[Kang]{Moon-Jin Kang}
\address[Moon-Jin Kang]{\newline Department of Mathematics, \newline The University of Texas at Austin, Austin, TX 78712, USA}
\email{moonjinkang@math.utexas.edu}

\newtheorem{theorem}{Theorem}[section]
\newtheorem{lemma}{Lemma}[section]

\newtheorem{proposition}{Proposition}[section]
\newtheorem{remark}{Remark}[section]

\newcommand{\beq}{\begin{equation}}
\newcommand{\eeq}{\end{equation}}

\newcommand{\bbr}{\mathbb R}
\newcommand{\bbs}{\mathbb S}
\newcommand{\bbn}{\mathbb N}

\newcommand{\bbp} {\mathbb P}
\newcommand{\bbt} {\mathbb T}

\newcommand{\eps}{\varepsilon}

\numberwithin{figure}{section}
\def\charf {\mbox{{\text 1}\kern-.30em {\text l}}}
\begin{document}
%%%%%%%%%%%%%%%%

\date{\today}

 \keywords{}

\thanks{\textbf{Acknowledgment.} I. M. Gamba is supported by the NSF under grants DMS-1109625, and NSF
RNMS (KI-Net) grant DMS11-07465, and M.-J. Kang is supported by Basic Science Research Program through the National Research Foundation of Korea funded by the Ministry of Education, Science and Technology (NRF-2013R1A6A3A03020506). The support from the Institute of Computational Engineering and Sciences at
the University of Texas Austin is gratefully acknowledged.}

\begin{abstract}
We study the global existence and uniqueness of weak solutions to kinetic Kolmogorov-Vicsek models that  can be considered  a non-local non-linear Fokker-Planck type equation describing the dynamics of individuals with orientational interactions. This model is derived from the discrete Couzin-Vicsek algorithm as mean-field limit \cite{B-C-C,D-M}, which governs the interactions of stochastic agents moving with a velocity of constant magnitude, i.e. the the corresponding velocity space for these type of Kolmogorov-Vicsek models are the unit sphere. Our analysis  for $L^p$ estimates and compactness properties take advantage of the orientational interaction property meaning that the velocity space is a compact manifold. 
\end{abstract}
\maketitle \centerline{\date}

%\tableofcontents

\section{Introduction}
\setcounter{equation}{0}
Recently, a variety of mathematical models capturing the emergent phenomena of self-driven agents have received extensive attention. In particular, the discrete Couzin-Vicsek algorithm (CVA) has been proposed as a model describing the interactions of agents moving with velocity of constant magnitude, and with angles measured from a reference direction (see \cite{A-H, C-K-J-R-F, G-C, Vicsek}).

This manuscript focuses on analytical issues for  the kinetic (mesoscopic) description associated to  the discrete  Couzin-Vicsek algorithm  with stochastic dynamics corresponding to Brownian motion on a sphere. More precisely, we consider  the corresponding  kinetic  Kolmogorov-Vicsek model describing stochastic particles with orientational interaction,

\begin{align}
\begin{aligned} \label{main}
&\partial_t f + \omega\cdot\nabla_x f = -\nabla_{\omega}\cdot(fF_o  ) + \mu\Delta_{\omega} f,\\
&F_o(x,\omega,t)= \nu(\omega\cdot\Omega(f)) (Id-\omega\otimes\omega)\Omega(f),\\
&\Omega(J)(x,t)=\frac{J(f)(x,t)}{|J(f)(x,t)|},\ \qquad J(f)(x,t)= \int_{U \times\bbs^{d-1}} K(|x-y|)\omega f (y,\omega, t) dy d\omega,\\
&f(x,\omega,0) = f_0(x,\omega),  \quad ~x\in U,~\omega\in\bbs^{d-1},~t>0,
\end{aligned}
\end{align}
%
%
%\begin{aligned} \label{main}
%&\partial_t f + \omega\cdot\nabla_x f = -\nabla_{\omega}\cdot(fF_o  ) + \mu\Delta_{\omega} f,\\
%&F_o(x,\omega,t)= \nu(\omega\cdot\Omega(f)) (Id-\omega\otimes\omega)\Omega(f),\\
%&\Omega(f)(x,t)=\frac{J(f)(x,t)}{|J(f)(x,t)|},\quad J(f)(x,t)= \int_{\bbr^{d}\times\bbs^{d-1}} K(|x-y|)\omega f (y,\omega, t) dy d\omega,\\
%&f(x,\omega,0) = f_0(x,\omega),  \quad ~x\in \bbr^d,~\omega\in\bbs^{d-1},~t>0.
%\end{aligned}
%\end{align}
\noindent
where $f = f(x,\omega,t)$ is the one-particle distribution function at position $x\in U$, velocity direction $\omega \in \bbs^{d-1}$ and time $t$. The spatial domain $U$ denotes either $\bbr^d$ or $\bbt^d$. The operators $\nabla_{\omega}$ and $\Delta_{\omega}$ denote the gradient and the Laplace-Beltrami operator on the sphere $\bbs^{d-1}$ respectively, and $\mu>0$ is a diffusion coefficient. The term $F_o(x,\omega,t)$ is the mean-field force  that governs the orientational interaction of self-driven particles by aligning them with the direction $\Omega(x,t)\in \bbs^{d-1}$ that depends on the flux $J(x,t)$. 

This mean-field force is also proportional to the interaction frequency $\nu$. Its reciprocal $\nu^{-1}$ represents the typical time-interval between two successive changes in the trajectory of the orientational swarm particle to accommodate the presence of other particles in the neighborhood.  The function  $K$ is an isotropic observation kernel around each particle and it is assumed to be integrable in $\bbr$. 

Following  Degond and Motsch in \cite{D-M}, the interaction frequency function $\nu$ is taken to be a positive function of $\cos\theta$, where $\theta$ is the angle between 
$\omega$ and $\Omega$. Such dependence of $\nu$ with respect to the angle $\theta$ represents  different  turning transition rates at different angles. Hence, the constitutive form of such  interaction frequency $\nu(\theta)$ is inherent to  species being modeled by orientational interactions. As in \cite{D-M}, we assume that $\nu(\theta)$ is a smooth and bounded function of its argument.

The kinetic Kolmogorov-Fokker-Planck type model with orientational interactions \eqref{main} was formally derived in \cite{D-M}  as a mean-field limit of the discrete Couzin-Vicsek algorithm (CVA) with stochastic dynamics. 
%This discrete model arises in the study of the interactions of individuals among animal societies such as fish schools moving with a velocity of constant magnitude. 
There, the authors  mainly  focused on the model \eqref{main} with the following local momentum $\tilde{J}$ instead of $J$  
\begin{align}
\begin{aligned} \label{main-0}
%&\partial_t f + \omega\cdot\nabla_x f = -\nabla_{\omega}\cdot(fF_o  ) + \mu\Delta_{\omega} f,\\
%&F_o(x,\omega,t)= \nu(\omega\cdot\Omega(f)) (Id-\omega\otimes\omega)\Omega(f),\\
&\Omega(\tilde J)(x,t)=\frac{\tilde{J}(f)(x,t)}{|\tilde{J}(f)(x,t)|},\ \qquad \tilde{J}(f)(x,t)= \int_{\bbs^{d-1}} \omega f (x,\omega, t) d\omega,\\
%&f(x,\omega,0) = f_0(x,\omega),  \quad ~x\in U,~\omega\in\bbs^{d-1},~t>0.
\end{aligned}
\end{align}
where $\tilde{J}$ was derived from ${J}$ in \eqref{main} by rescaling the kernel $K$ in time and spatial variables. Such scaling describes dynamics for  solutions to \eqref{main} at large time and length scales compared with scales of the individuals. 

In the current manuscript, we focus on existence and uniqueness properties  of solutions to  both models, with $J(f)$ as defined in \eqref{main} and  with $\tilde{J}(f)$ as defined in \eqref{main-0}.

In fact, since $J$ with the kernel $K=\delta_0$(Dirac mass) is exactly $\tilde{J}$, it is enough to  show  global existence and uniqueness of weak solutions to models \eqref{main} in an appropriate space, to be specified in Section 2. These results are easily applied to  $\tilde{J}$ as  in \eqref{main-0}. 

\medskip

The classical Vicsek model have received extensive attention in the last few years concerning the rigor of mathematical studies of its mean-field land  hydrodynamic limits as well as  phase transition development. More specifically,  Bolley, Ca$\tilde{\mbox{n}}$izo and Carrillo have rigorously justified  a mean-field limit in \cite{B-C-C} when the force term acting on the particles is not normalized, i.e., $\nu\Omega(x,t)$ replaced by just $J(x,t)$ in force term $F_o$. This modification leads to the appearance of phase transitions from disordered states at low density to aligned (ordered) states at high densities. Such phase transition  problem has been studied in \cite{A-H, C-K-J-R-F, D-F-L-1, D-F-L-2, F-L, G-C}. In addition, issues on hydrodynamic descriptions of classical Vicsek model have been discussed in  \cite{D-F-L-1, D-F-L-2, D-M, D-M-2, D-Y, F}. We also refer to \cite{Bo-Ca, D-D-M, H-J-K} concerning related issues.

%Although the Vicsek model has been studied via a variety of scales from microscopic level to macroscopic level, 

Up to date, there are few results on existence theory of true kinetic descriptions. Frouvelle and Liu \cite{F-L} have shown the well-posedness in the space-homogeneous case of \eqref{main-0} with the regular force field $(Id-\omega\otimes\omega)\tilde{J}$ instead of $(Id-\omega\otimes\omega)\Omega(\tilde J)$. There, they have provided the convergence rates towards equilibria by using the Onsager free energy functional and Lasalle's invariance principle, and their results have been applied in \cite{D-F-L-1}. Very recently, Figalli, Morales and the second author \cite{F-K-M} have shown the well-posedness in the space-homogeneous case of \eqref{main-0}, and the convergence of solutions towards steady states, based on the gradient flow approach (see for example \cite{F-G, J-K-O}).

 On the other hand, the authors in \cite{B-C-C} have shown existence of weak solutions for the space-inhomogeneous equation for a force field $F_o$  given by  the difference between spatial convolutions of mass and momentum with a bounded Lipschitz kernels $K$,  namely $\omega K*_{x} \rho - K*_{x} J$,  instead of   $\nu\Omega$ as considered in this manuscript. Such a choice of  force field has a regularizing effect for spatial variable compared to our case $\nu\Omega$ which deals with stronger non-linearities.\\

\smallskip

This manuscript  is mainly devoted to showing the existence and uniqueness properties of weak solutions to the kinetic Kolmogorov-Vicsek type model \eqref{main}.  A difficulty in our analysis arises from the fact that $\Omega(J)$ in the alignment force term of \eqref{main} is undefined at 
${J}(f)$ becomes $0$.  So we restrict the problem of finding  global weak solutions to \eqref{main} to a subclass of solutions with the non-zero local momentum, i.e. ${J}(f)\neq 0$. 

In the next section, we briefly present some known results for kinetic models with orientational interactions, \eqref{main} and \eqref{main-0}, which give a heuristic justification for the a priori non-zero  assumption on $J(f)$ to be stated in our main result.  Section 3 presents a priori estimates and the compactness lemma, which play crucial roles in the main proof of existence of weak solution in the next section. Section 4 deals with the construction of weak solutions to  \eqref{main} by means of first, introducing  an $\eps$-regularized problem, for an arbitrary parameter $\eps>0$ modifying the alignment force    $\Omega(J)$ uniformly bounded in $\eps$. We then solve the $\eps$-regularized problem of \eqref{main} constructing a sequence of functions $\{ f_{n,\eps}\}_{n \geq 1}$ that converges to the solution $f_\eps$. Finally we show that, in within the class of solutions satisfying ${J}(f)\geq 0$, there is a subsequence $f_{\eps_k}$ converging to $f$, solving \eqref{main}. Section 5 is devoted to the proof of the uniqueness of weak solutions in a periodic spatial domain $U=\bbt^d$ under the additional constraint ${J}(f)\geq \alpha >0$.

\section{Preliminaries and Main results}
\setcounter{equation}{0}
In this section, we briefly review how the kinetic Kolmogorov-Viscek equations,  \eqref{main} and \eqref{main-0} can be formally
derived from the discrete Couzin-Vicsek algorithm model \cite{D-M} with stochastic dynamics. Then we provide our main result and useful formulations.

\subsection{Kinetic Kolmogorov-Vicsek models}
Following \cite{D-M}, the kinetic Kolmogorov-Vicsek model considered in \eqref{main}  is derived from the classical discrete Vicsek formulation modeling  Brownian motion of the sphere $\bbs^{d-1}$  given by  the following stochastic differential equations for $1\le i \le N$,

\begin{align}
\begin{aligned} \label{SDE}
& dX_i = \omega_i dt,\\
& d\omega_i = (Id-\omega_i\otimes\omega_i) \nu(\omega_i\cdot \bar{\Omega}_i) \bar{\Omega}_i dt + \sqrt{2\mu} 
(Id-\omega_i\otimes\omega_i)\circ dB^i_t,\\
& \bar{\Omega}_i = \frac{\bar{J_i} }{ |\bar{J_i}|},\quad \bar{J_i} = \sum_{j, ~|X_j-X_i|\le R} \omega_j.
\end{aligned}
\end{align}
Here, the neighborhood of the $i$-th particle is the ball centered at $X_i\in \bbr^d$ with radius $R>0$. The velocity director $\omega_i\in \bbs^{d-1}$ of the $i$-th particle tends to be aligned with the director $\Omega_i$ of the average velocity of the neighboring particles with noise $B^i_t$ standing for $N$ independent standard Brownian motions on $\bbr^{d}$ with intensity $\sqrt{2\mu}$. Then, its projection $(Id-\omega_i\otimes\omega_i)\circ dB^i_t$ represents the contribution of a Brownian motion on the sphere $\bbs^{d-1}$, which should be understood in the Stratonovich sense.  We refer to \cite{H} for a detailed description on Brownian motions on Riemannian manifolds.
We note that  the first  term in $d\omega_i$   is the sum of smooth binary interactions with identical speeds, whereas there is no constraint on the velocity in the Cucker-Smale model \cite{C-S}. In addition the interaction frequency (weight) function $\nu(\omega_i\cdot \Omega_i)$ depends on the angle between $\omega_i$ and $\Omega_i$, parametrized by $\cos \theta_i = \omega_i\cdot \Omega_i$.

From the individual-based model \eqref{SDE}, the corresponding kinetic mean-field limit \eqref{main} was proposed in \cite{B-C-C, D-M}, as the number of particles $N$ tends to infinity. Notice that $\mu$ in \eqref{main} corresponds to the diffusive coefficient associated to the Brownian motion on the sphere  $\bbs^{d-1}$. 
 
The reduced model \eqref{main} with the modified definition of setting $J=\tilde J$ as in \eqref{main-0} was proposed in  \cite{D-M} by the following  scaling argument. Considering the system dynamics  at large times and length scales compared with those scales of individuals by the dimensionless rescaled varaibles $\tilde{x} = \eps x, \tilde{t} = \eps t$ with $\eps \ll 1$,  it makes the interactions to become local and  aligned the particle velocity into the direction of the local particle flux. This interaction term is balanced at leading order $\eps$ by the diffusion term. 
 
Notice that $\Omega(f)$ in \eqref{main} is undefined when ${J}(f)$ becomes $0$. Because of this issue, we study in this manuscript  the existence of weak solutions to \eqref{main} for the subclass of solutions with non-zero local momentum, i.e. ${J}(f)\neq 0$. As shown in \cite{D-M}, since $\omega$ is not a collisional invariant of operator $Q$, the momentum is not conserved. Thus, it is not straightforward to get $J(f)(x,t)\neq 0$ for all $(x,t)$ from imposing non-zero initial momentum, i.e. $J(f)(x,0)\neq 0$ for all $x$. Moreover, there is no canonical entropy for the type of the kinetic equations as in \eqref{main}. Due to these analytical difficulties, we  heuristically justify our constraint $J(f)\neq 0$ by observing equilibria of \eqref{main-0} in the three dimensional case, which has been studied in \cite{D-M} as follows.

 For the classification of equilibria in the  $d=3$ dimensional case, we recall the  the Fisher-von Mises distribution, given by
\[
M_{\Omega}(\omega) = \frac{1}{\int_{\bbs^{2}}\exp (\frac{\sigma(\omega\cdot \Omega)}{\mu})d\omega} \exp \Big(\frac{\sigma(\omega\cdot \Omega)}{\mu}\Big)
\]
for a given unit vector $\Omega\in \bbs^{2}$, where $\sigma$ denotes an antiderivative of $\nu$, i.e. $\frac{d\sigma}{d\tau}(\tau)=\nu(\tau)$. Since $\nu$ is positive, $\sigma$ is an increasing function and then  $M_{\Omega}$ is maximal at $\omega\cdot \Omega = 1$, that is for $\omega$ pointing in the direction of $\Omega$. Therefore, $\Omega$ plays the same role as the averaged velocity  in the classical Maxwellian equilibria for classical kinetic models of rarefied gas dynamics with velocities defined in all space. 
The diffusion constant $\mu$ corresponds to the temperature strength, which measures the spreading of the equilibrium state about the average direction $\Omega$. The present model has  a constant diffusion $\mu$ that is in contrast with the classical gas dynamics where the temperature is a thermodynamical variable whose evolution is determined by the energy balance equation.

Using the Fisher-von Mises distribution, the operator $Q$ and equilibria of \eqref{main-0} are expressed as follows.

\begin{lemma}
\emph{\cite{D-M}}
(i) The operator $Q(f)$ can be written as
\[
Q(f) = \mu \nabla_{\omega} \cdot \Big[M_{\Omega(f)} \nabla_{\omega}\Big( \frac{f}{M_{\Omega(f)}}\Big) \Big].
\]
(ii) The equilibria, i.e. solutions $f(\omega)$ satisfying Q(f) = 0, form a three dimensional manifold $\mathcal{E}$ given by
\[
\mathcal{E} =  \{ \rho M_{\Omega} (\omega) ~|~ \rho >0,~ \Omega\in\bbs^2 \},
\]
\ 
where $\rho$ is the total mass and $\Omega$ is  the flux director of $\rho M_{\Omega} (\omega)$, that is,
\begin{align*}
\begin{aligned} 
&\rho = \int_{\bbs^2} \rho M_{\Omega} (\omega) d\omega ,\quad \Omega = \frac{\tilde{J}(\rho M_{\Omega})}{\Big|\tilde{J}(\rho M_{\Omega}) \Big|},\\
&\tilde{J}(\rho M_{\Omega}):=\int_{\bbs^2} \rho M_{\Omega} (\omega) \omega d\omega = \rho c(\mu) \Omega,
\end{aligned}
\end{align*}
with 
\begin{align*}
\begin{aligned} 
c(\mu)=\frac{\int_0^{\pi}\cos\theta \exp \Big(\frac{\sigma(\cos\theta)}{\mu} \Big)\sin\theta d\theta }
{\int_0^{\pi} \exp \Big(\frac{\sigma(\cos\theta)}{\mu} \Big)\sin\theta d\theta }.
\end{aligned}
\end{align*}
\end{lemma}
We note that $c(\mu)\rightarrow 1$ as $\mu\rightarrow 0$, and $c(\mu)\rightarrow 0$ as $\mu\rightarrow \infty$. This means that the local momentum $\tilde{J}(\rho M_{\Omega})$ of the equilibrium solution $f=\rho M_{\Omega}$ is not zero as long as the diffusion strength $\mu$ is not sufficiently large compared to orientational interaction. Consequently, it is expected that moderate values of $\mu$ would yield non-zero local momentum $\tilde{J}(f)$ for solutions $f$ near the Von Mises equilibria.

%Therefore, if the diffusion constant $\mu$ is not large, it is reasonable that the local momentum $\tilde{J}(f)$ of our solution $f$ is assumed to be nonzero at least near the equilibrium.  

\subsection{Main result}
We state now the main results for global existence of weak solutions to equations \eqref{main}.\\ 
We first introduce the following notations for simplification.\\
 
$\bullet$ {\bf Notation} : We denote by $D:=U\times\bbs^{d-1}$, and  by $\bbp_{\omega^{\perp}}:=Id-\omega\otimes\omega$, as 
the mapping $v\mapsto (Id-\omega\otimes\omega)v$ is the projection of the vector $v$ onto the normal plane to $\omega$.\\

$\bullet$ {\bf Hypotheses ($\mathcal{H}$)} :
As stated earlier, we assume that $\nu(\cdot)$ is a smooth and bounded function of its argument and $K(|\cdot|)\in L^{1}(U)$.
Moreover, in order to avoid   $\Omega(f)$ to be undefined, we impose a priori assumptions stating that 
  the weak solutions $f$ of \eqref{main} belong to an admissible class
\beq\label{assume}
\mathcal{A}:=\{f~|~ {J}(f)(x,t)\neq 0,\quad \forall x\in U,~t > 0 \}.
\eeq

\begin{theorem}[\bf Existence for spatial domains $ U$, being  either $\bbr^d$ or $\bbt^d$] \label{thm-exist}
Assume $(\mathcal{H})$ and $f_0$ satisfies
\beq\label{initial-con}
f_0\in (L^1\cap L^{\infty})(D)\quad \mbox{and}\quad f_0\ge 0.
\eeq
Then, for a given $T>0$, the equation \eqref{main} has a weak solution $f$, which satisfies 
\begin{align}
\begin{aligned}\label{regularity}
& f\ge 0,\\
&f\in C(0,T;L^1(D))\cap L^{\infty}(D\times (0,T)),\\
&\nabla_{\omega} f \in L^{2} (D\times (0,T)).
\end{aligned}
\end{align}
and the following weak formulation: for any $\phi\in C^{\infty}_c (D\times [0,T))$,
\begin{align}
\begin{aligned}\label{weak-form}
&\int_0^t\int_{D}f \partial_t\phi +f\omega\cdot\nabla_x \phi  + fF_o\cdot\nabla_{\omega} \phi - \mu\nabla_{\omega} f \cdot\nabla_{\omega}\phi dxd\omega ds\\
&\hspace{3cm} +\int_{D} f_0 \phi(0,\cdot) dxd\omega = 0,\\
&F_o(x,\omega,t)= \nu(\omega\cdot\Omega(f)) \bbp_{\omega^{\perp}}\Omega(f).
\end{aligned}
\end{align}
Moreover, the weak solution $f$ satisfies the following estimate
\beq\label{p-est-0}
\|f\|_{L^{\infty}(0,T; L^p(D))} +\frac{2\mu (p-1)}{p} \|\nabla_{\omega}f^{\frac{p}{2}}\|_{L^{2}(D\times (0,T))}^{\frac{2}{p}} \le e^{CT\frac{p}{p-1}}\|f_0\|_{L^p(D)},
\eeq
 for any $1\le p<\infty$, and 
\beq\label{bound-est-0}
\|f\|_{L^{\infty}(D\times (0,T))} \le e^{CT}\|f_0\|_{L^{\infty}(D)}.
\eeq
\end{theorem} 

\begin{remark}
The proof of Theorem \ref{thm-exist} is based on energy methods, where the diffusion term $\mu\Delta_{\omega}f$ plays a crucial role. Yet  the strength $\mu>0$ does not essentially affect the proof of existence. Therefore, without loss of generality, from now on we set $\mu=1$.
\end{remark}

{We next present uniqueness  of weak solutions being constructed in Theorem \ref{thm-exist}, only for  periodic domains $U=\bbt^d$, together with the following subclass \
\[
\mathcal{A}_{\alpha}:=\{f~|~\exists~\alpha>0~s.t. ~|J(f)(x,t)|>\alpha,~\forall (x,t)\in\bbt^d\times (0,T) \},
\]
which is more restrictive than \eqref{assume}. Indeed this class corresponds to the  subclass
of weak solutions to the initial value problem \eqref{main}, with uniformly bounded below speed when solved in a spatial torus domain.  \
\

\begin{theorem}[\bf Uniqueness for periodic spatial domains  $\bbt^d$]\label{thm-unique} 
Assume $(\mathcal{H})$ and \eqref{initial-con}. Then for a given $T>0$, the periodic boundary problem of \eqref{main} has a unique weak solution $f$ in the subclass $\mathcal{A}_{\alpha}$.
\end{theorem}
\begin{remark}
Our proof for uniqueness takes advantage of a uniformly positive lower bound $\alpha$ of $J(f)$ in order to control $\Omega(f)$, consequently restrict to periodic domain $\bbt^d$. Indeed, imposing that $J(f)\ge\alpha>0$ for all $x\in \bbr^d$ results in an infinite mass $\int_{\bbr^d\times\bbs^{d-1}} f dxd\omega =\infty$, due to
\begin{align*}
\begin{aligned}
\infty= \int_{\bbr^d}J(f)dx &\le  \int_{\bbr^d\times\bbs^{d-1}} |K*_{x}f| dxd\omega \\
&\le   \int_{\bbs^{d-1}} \|K*_{x}f\|_{L^1(\bbr^d)}d\omega \le  \|K\|_{L^1} \int_{\bbr^d\times\bbs^{d-1}} f dxd\omega,
\end{aligned}
\end{align*}
\end{remark}
}

\medskip

\subsection{Formulas for Calculus on sphere} 
We start recalling some useful formulas on the sphere $\bbs^{d-1}$ which are extensively used in this paper.\\
Let $F$ be a vector-valued function and $f$ be a scalar-valued function. The following formula, as a analogous of the integration by parts, holds
\beq\label{formula-0}
\int_{\bbs^{d-1}} f\nabla_{\omega}\cdot F d\omega =  -\int_{\bbs^{d-1}} F\cdot(\nabla_{\omega}f -2\omega f) d\omega.
\eeq

By the definition of the projection operator $\bbp_{\omega^{\perp}}$, it follows that
\begin{align}
\begin{aligned}\label{formula-1}
&\bbp_{\omega^{\perp}}\omega = 0,\quad \bbp_{\omega^{\perp}}\nabla_{\omega} f =\nabla_{\omega} f,\\
& \bbp_{\omega^{\perp}} u\cdot v= \bbp_{\omega^{\perp}} v\cdot u ,
\end{aligned}
\end{align}
for any scalar-valued function $f$, and vectors $u$ and $v$.\\

In addition, for any constant vector $v\in\bbr^d$, we have
\begin{align}
\begin{aligned}\label{formula-2}
&\nabla_{\omega}(\omega\cdot v) = \bbp_{\omega^{\perp}} v,\\
&\nabla_{\omega}\cdot(\bbp_{\omega^{\perp}}v) = -(d-1) \omega\cdot v.
\end{aligned}
\end{align}
These formulas can be easily derived classical calculus on spherical coordinates  (see \cite{F-L, O-T}).

\section{A priori estimates and compactness lemma}

\setcounter{equation}{0}
The following Lemma provides a priori estimates in $L^\infty(0,T; L^p(U)), 1\leq p\leq \infty$ for solutions to the initial value problem for the kinetic equation below. The subsequent 
Lemma~\ref{lem-compact} provides a  compactness  tool needed for the existence result proof of Theorem \ref{thm-exist}. 

\begin{lemma}\label{lem-priori}
Assume that $f_0$ satisfies \eqref{initial-con}, and $f$ is a smooth solution to the equation
\begin{align}
\begin{aligned} \label{main-compact}
&\partial_t f+ \omega\cdot\nabla_x f= -\nabla_{\omega}\cdot\Big( f\nu(\omega\cdot \Omega)\bbp_{\omega^{\perp}} \Omega  \Big) + \Delta_{\omega} f,\\
&f(x,\omega,0) = f_0(x,\omega),
\end{aligned}
\end{align}
where $\Omega~:~U\times\bbr_{+} \rightarrow \bbr^d$ is a bounded vector-valued function of $(x,t)$. \\
 Then, for any $1\le p<\infty$,
\beq\label{p-est}
\|f\|_{L^{\infty}(0,T; L^p(D))} +\frac{2(p-1)}{p} \|\nabla_{\omega}f^{\frac{p}{2}}\|_{L^{2}(D\times (0,T))}^{\frac{2}{p}} \le e^{CT\frac{p}{p-1}}\|f_0\|_{L^p(D)}. 
\eeq
In particular, if $p=\infty$, then 
\beq\label{bound-est}
\|f\|_{L^{\infty}(D\times (0,T))} \le e^{CT}\|f_0\|_{L^{\infty}(D)}.
\eeq
\end{lemma}

\smallskip

\begin{proof}
First of all, for any $1\le p<\infty$, it follows from \eqref{main} that
\begin{align}
\begin{aligned}
\frac{d}{dt}\int_{D} f^p dxd\omega &= - p \int_{D} f^{p-1} \nabla_{\omega}\cdot(f \nu(\omega\cdot\Omega)\bbp_{\omega^{\perp}}\Omega) dxd\omega + p\int_{D} f^{p-1} \Delta_{\omega} f dx d\omega \\
& =: I_1 + I_2\ .  \label{I_1}
\end{aligned}
\end{align}
Using formula \eqref{formula-0} and $\omega\cdot\nabla_{\omega} f =0$, we have 
\begin{align*}
\begin{aligned}
I_2&= -p(p-1)\int_{D} f^{p-2}\nabla_{\omega} f \cdot \nabla_{\omega} f dx d\omega + 2p\int_{D} f^{p-1} \omega\cdot \nabla_{\omega} f dx d\omega\\
&=-\frac{4(p-1)}{p} \int_{D} | \nabla_{\omega} f^{\frac{p}{2}}|^2 dx d\omega. 
\end{aligned}
\end{align*}

Next, by formula \eqref{formula-2},   the term $I_1$ from \eqref{I_1} is estimated as follows 
\begin{align*}
\begin{aligned}
I_1&= - p \int_{D} f^{p-1} \Big( \nu(\omega\cdot\Omega)\nabla_{\omega}  f\cdot \bbp_{\omega^{\perp}}\Omega 
+f \nu^{\prime}(\omega\cdot\Omega) |\bbp_{\omega^{\perp}}\Omega|^2 
- (d-1) f \nu(\omega\cdot\Omega)\omega\cdot\Omega \Big)dxd\omega\\
&\le p\| \nu(\omega\cdot\Omega)\|_{L^{\infty}} \int_{D} f^{p-1}|\nabla_{\omega}  f|dxd\omega
+p\| \nu^{\prime}(\omega\cdot\Omega)\|_{L^{\infty}} \int_{D} f^{p}dxd\omega\\
&\quad+p(d-1)\| \nu(\omega\cdot\Omega)\|_{L^{\infty}} \int_{D} f^{p}dxd\omega.
\end{aligned}
\end{align*}

In addition, using H$\ddot{\mbox{o}}$lder's inequality, the first integral in the right hand side above can be estimated by
\begin{align*}
\begin{aligned}
\int_{D} f^{p-1}|\nabla_{\omega}  f|dxd\omega &\le\Big(\int_{D} f^{p}dxd\omega\Big)^{1/2}
\Big(\int_{D} f^{p-2}|\nabla_{\omega}  f|^2 dxd\omega\Big)^{1/2}\\
&=\frac{2}{p}\Big(\int_{D} f^{p}dxd\omega\Big)^{1/2}
\Big(\int_{D} | \nabla_{\omega} f^{\frac{p}{2}}|^2 dx d\omega\Big)^{1/2}.
\end{aligned}
\end{align*}
Then, we have
\begin{align*}
\begin{aligned}
I_1\le \frac{2(p-1)}{p} \int_{D} | \nabla_{\omega} f^{\frac{p}{2}}|^2 dx d\omega + C(\frac{p}{p-1}+p)\int_{D} f^{p}dxd\omega.
\end{aligned}
\end{align*}
Finally, combining the estimates above for both $I_1$ and $I_2$, we get
\[
\frac{d}{dt}\int_{D} f^p dxd\omega + \frac{2(p-1)}{p} \int_{D} | \nabla_{\omega} f^{\frac{p}{2}}|^2 dx d\omega 
\le C(\frac{p}{p-1}+p)\int_{D} f^{p}dxd\omega,
\]
which yields a Gronwall type inequality
\[
\frac{d}{dt} \|f\|_{L^p(D)} \le C\frac{p}{p-1} \|f\|_{L^p(D)}.
\]
Therefore,
\[
\|f\|_{L^{\infty}(0,T; L^p(D))} \le e^{CT\frac{p}{p-1}}\|f_0\|_{L^p(D)},
\]
which implies the $L^p$ estimate in \eqref{p-est}. Hence, taking $p\to\infty$, yields the  $L^{\infty}$ bound \eqref{bound-est}.
\end{proof}

\begin{remark}
The boundedness of the alignment  vector $\Omega$ is essential for  the proof of Lemma \ref{lem-priori},  and the a priori estimates \eqref{p-est} and \eqref{bound-est} still hold for $\Omega= \Omega(f)$ bounded for any $f$. 
\end{remark}

\medskip

The following lemma provides the compactness property that ensures the strong $L^p$ convergence of solutions to the initial value problem associated to linear equation \eqref{main-compact2}.  Such strong compactness property relies on the boundedness of both the force term and velocity space (notice that the velocity variable would be unbounded, we would have to use the celebrated velocity averaging lemma \cite{K-M-T,P-S}).  As mentioned earlier, the compactness property obtained from next lemma is crucial for the existence proof of  Theorem \ref{thm-exist}.  

\begin{lemma}\label{lem-compact}
Assume that $f_0$ satisfies \eqref{initial-con}, and $f_n$ is a smooth solution to 
\begin{align}
\begin{aligned} \label{main-compact2}
&\partial_t f_n+ \omega\cdot\nabla_x f_n= -\nabla_{\omega}\cdot\Big( f_n\nu(\omega\cdot F_n)\bbp_{\omega^{\perp}} F_n  \Big) + \Delta_{\omega} f_n,\\
&f_n(x,\omega,0) = f_0(x,\omega),
\end{aligned}
\end{align}
where $F_n~:~U\times\bbr_{+} \rightarrow \bbr^d$ is a given function of $(t, x)$. \\
If the sequence $(F_n)$ is bounded in $L^{\infty}(U\times (0,T))$, then there exists a limit function $f$ such that, up to a subsequence, 
\[
f_n \rightarrow f \quad\mbox{as}~n\rightarrow \infty ~\mbox{in}~L^{p}(D\times (0,T))\cap L^{2}(U\times (0,T) ; H^1(\bbs^{d-1}))'\, , \quad
1\leq p < \infty.
\]
Moreover, the associated sequence
\[
{J}_n:= \int_{D} K(|x-y|)\omega f_n (y,\omega, t) dy d\omega 
\]
strongly converges to the corresponding limit ${J}$ in $L^{p}(U\times (0,T))$
where
\[
{J}:= \int_{D} K(|x-y|)\omega f (y,\omega, t) dy d\omega.
\]
\end{lemma}
\begin{proof}
Since the sequence $(F_n)$ is bounded in $L^{\infty}(U\times (0,T))$, there exists $F\in L^{\infty}(U\times (0,T))$ such that, up to a subsequence,
\beq\label{weak-co}
F_n \rightharpoonup F\quad \mbox{weakly}-*~\mbox{in}~L^{\infty}(U\times (0,T)). 
\eeq
Let $f$ be a solution of \eqref{main-compact2} corresponding to the limiting $F$.  Then, the following identity holds
\begin{align}\label{linear-error}
\begin{aligned}
\partial_t (f_n-f) + \omega\cdot\nabla_x (f_n-f) &= -\nabla_{\omega}\cdot\Big( (f_n-f)\nu(\omega\cdot F_n)\bbp_{\omega^{\perp}} F_n  \Big)\\
&\quad -\nabla_{\omega}\cdot\Big( f (\nu(\omega\cdot F_n)-\nu(\omega\cdot F))\bbp_{\omega^{\perp}} F_n) \Big)\\ 
&\quad -\nabla_{\omega}\cdot\Big( f\nu(\omega\cdot F) \bbp_{\omega^{\perp}} (F_n- F ) \Big) + \Delta_{\omega} (f_n -f).
\end{aligned}
\end{align}

Next, for any fixed $p\in [1,\infty)$, multiplying the above equation by $p(f_n-f)^{p-1}$ and  integrating over $D$ yields the identity
\begin{align}
\begin{aligned}\label{identity-Ji}
&\frac{d}{dt}\int_{D} (f_n-f)^p dxd\omega \\
&\quad=- p \int_{D} (f_n-f)^{p-1} \nabla_{\omega}\cdot\Big( (f_n-f)\nu(\omega\cdot F_n)\bbp_{\omega^{\perp}} F_n  \Big)dxd\omega \\
&\qquad - p \int_{D} (f_n-f)^{p-1} \nabla_{\omega}\cdot\Big( f (\nu(\omega\cdot F_n)-\nu(\omega\cdot F)) \bbp_{\omega^{\perp}} F_n \Big)dxd\omega \\
&\qquad - p \int_{D} (f_n-f)^{p-1}\nabla_{\omega}\cdot\Big( f \nu(\omega\cdot F)\bbp_{\omega^{\perp}} (F_n- F ) \Big)dxd\omega\\
&\qquad + p \int_{D} (f_n-f)^{p-1} \Delta_{\omega} (f_n -f)dxd\omega\\
&=:\mathcal{J}_1 + \mathcal{J}_2+\mathcal{J}_3+\mathcal{J}_4.
\end{aligned}
\end{align}

We first estimate  the term $\mathcal{J}_1$ using the same arguments as  the ones used in  Lemma \ref{lem-priori}  in order to estimate $I_1$. Indeed, 
\begin{align*}
\begin{aligned}
\mathcal{J}_1 &= - p \int_{D} (f_n-f)^{p-1} \Big( \nu(\omega\cdot F_n)\nabla_{\omega}  (f_n-f)\cdot \bbp_{\omega^{\perp}}F_n 
+(f_n-f) \nu^{\prime}(\omega\cdot F_n) |\bbp_{\omega^{\perp}}F_n|^2\\ 
&\qquad- (d-1) (f_n-f) \nu(\omega\cdot F_n)\omega\cdot F_n \Big)dxd\omega\\
&\le p\| \nu(\omega\cdot F_n)\|_{L^{\infty}}\|F_n\|_{L^{\infty}} \int_{D} (f_n-f)^{p-1}|\nabla_{\omega} (f_n-f)|dxd\omega\\
&\quad+p\| \nu^{\prime}(\omega\cdot F_n)\|_{L^{\infty}} \|F_n\|_{L^{\infty}}^2\int_{D} (f_n-f)^{p}dxd\omega\\
&\quad+p(d-1)\| \nu(\omega\cdot F_n)\|_{L^{\infty}}\|F_n\|_{L^{\infty}} \int_{D} (f_n-f)^{p}dxd\omega.\\
&\le \frac{2(p-1)}{p} \int_{D} | \nabla_{\omega}  (f_n-f)^{\frac{p}{2}}|^2 dx d\omega + \frac{Cp^2}{p-1}\int_{D} (f_n-f)^{p}dxd\omega.
\end{aligned}
\end{align*}
Similarly, $\mathcal{J}_4 $ is also estimated as done for ${I}_2$ in  the proof of Lemma \ref{lem-priori}, 
\[
\mathcal{J}_4 = -\frac{4(p-1)}{p} \int_{D} | \nabla_{\omega} (f_n-f)^{\frac{p}{2}}|^2 dx d\omega.
\]
Hence, gathering these two last estimates, identity \eqref{identity-Ji} yields the following estimate
\begin{align*}
\begin{aligned}
\frac{d}{dt}\int_{D} (f_n-f)^p dxd\omega \le C\int_{D} (f_n-f)^p dxd\omega-\frac{2(p-1)}{p} \int_{D} | \nabla_{\omega} (f_n-f)^{\frac{p}{2}}|^2 dx d\omega + \mathcal{J}_2+\mathcal{J}_3.
\end{aligned}
\end{align*}

Next, since $f_n=f$ at $t=0$, then applying the Gronwall's inequality to the above inequality, it holds that for any $0<t\le T$,
\begin{align*}
\begin{aligned}
\int_{D} (f_n-f)^p dxd\omega +\frac{2(p-1)}{p} \int_0^t \int_{D} | \nabla_{\omega} (f_n-f)^{\frac{p}{2}}|^2 dx d\omega ds \le e^{CT}\int_0^t (\mathcal{J}_2+\mathcal{J}_3 )(s) ds.
\end{aligned}
\end{align*}

The terms $\mathcal{J}_2$  and $\mathcal{J}_3$ can be rewritten using the calculaus on the sphere formulas \eqref{formula-2} as follows. 
First, note that the term   $\mathcal{J}_2$ satisfies the identity
\begin{align*}
\begin{aligned}
\mathcal{J}_2&= - p \int_{D} (f_n-f)^{p-1} \Big[ (\nu(\omega\cdot F_n)-\nu(\omega\cdot F))\nabla_{\omega} f\cdot \bbp_{\omega^{\perp}}F_n \\
&\qquad +f (\nu^{\prime}(\omega\cdot F_n)F_n-\nu^{\prime}(\omega\cdot F)F)\cdot \bbp_{\omega^{\perp}}F_n\\ 
&\qquad- (d-1) f (\nu(\omega\cdot F_n)-\nu(\omega\cdot F))\omega\cdot F_n \Big]dxd\omega\\
&= - p \int_{D} (f_n-f)^{p-1} \Big[ \nu^{\prime}(\omega\cdot F^*_n) \omega\cdot (F_n -F)\nabla_{\omega} f\cdot \bbp_{\omega^{\perp}}F_n \\
&\qquad +f \Big(\nu^{\prime}(\omega\cdot F_n)(F_n-F)+ \nu^{\prime\prime}(\omega\cdot F^{**}_n)\omega\cdot(F_n-F) F\Big)\cdot \bbp_{\omega^{\perp}}F_n\\ 
&\qquad- (d-1) f  \nu^{\prime}(\omega\cdot F^*_n) \omega\cdot (F_n -F)\omega\cdot F_n \Big]dxd\omega\\
&= - p \int_{D} (f_n-f)^{p-1} \Big[ \nu^{\prime}(\omega\cdot F^*_n)\nabla_{\omega} f\cdot \bbp_{\omega^{\perp}}F_n\omega \\
&\qquad +f \nu^{\prime}(\omega\cdot F_n)\bbp_{\omega^{\perp}}F_n+ f\nu^{\prime\prime}(\omega\cdot F^{**}_n)F\cdot\bbp_{\omega^{\perp}}F_n \omega  \\ 
&\qquad- (d-1) f  \nu^{\prime}(\omega\cdot F^*_n)\omega\cdot F_n \omega \Big]\cdot (F_n -F) dxd\omega,
\end{aligned}
\end{align*}
where $F^*_n$ and $F^{**}_n$ are some bounded functions due to the mean value theorem property, depending solely on the known bounded functions $F_n(x,t)$ and its limit $F$ defined in 
\eqref{weak-co}.\\
Similarly, also by the identities in \eqref{formula-2}, the term  $\mathcal{J}_3$ satisfies the identity
\begin{align*}
\begin{aligned}
\mathcal{J}_3&= - p \int_{D} (f_n-f)^{p-1} \Big[ \nu(\omega\cdot F)\nabla_{\omega} f\cdot \bbp_{\omega^{\perp}}(F_n-F)
+f \nu^{\prime}(\omega\cdot F)\bbp_{\omega^{\perp}}F\cdot \bbp_{\omega^{\perp}}(F_n-F) \\
&\qquad - (d-1) f \nu(\omega\cdot F)\omega\cdot (F_n-F) \Big]dxd\omega\\
&= - p \int_{D} (f_n-f)^{p-1} \Big[ \nu(\omega\cdot F)\nabla_{\omega} f +f \nu^{\prime}(\omega\cdot F)\bbp_{\omega^{\perp}}F
  - (d-1)f \nu(\omega\cdot F)\omega \Big]\cdot (F_n-F) dxd\omega.\\
\end{aligned}
\end{align*}

Thus, we get the {\sl weighted} estimate
\begin{align}
\begin{aligned}\label{vanish}
&\|f_n-f\|_{L^{p}(D)}^p +\frac{4(p-1)}{p} \int_0^T \int_{D} | \nabla_{\omega} (f_n-f)^{\frac{p}{2}}|^2 dx d\omega ds \\
&\le e^{CT} \int_0^T \int_{D} \Phi(x,w,s)\, \cdot  (F_n-F) dxd\omega ds,
\end{aligned}
\end{align}
where the weight function, given by  
\begin{align*}
\begin{aligned}
\Phi(x,w,s) &=  -p(f_n-f)^{p-1} \Big[\nu^{\prime}(\omega\cdot F^*_n)\nabla_{\omega} f\cdot \bbp_{\omega^{\perp}}F_n\omega +f \nu^{\prime}(\omega\cdot F_n)\bbp_{\omega^{\perp}}F_n+ f\nu^{\prime\prime}(\omega\cdot F^{**}_n)F\cdot\bbp_{\omega^{\perp}}F_n \omega  \\ 
&\qquad- (d-1) f  \nu^{\prime}(\omega\cdot F^*_n)\omega\cdot F_n \omega + \nu(\omega\cdot F)\nabla_{\omega} f +f \nu^{\prime}(\omega\cdot F)\bbp_{\omega^{\perp}}F
  - (d-1)f \nu(\omega\cdot F)\omega \Big].
\end{aligned}
\end{align*}
 is shown to satisfy   $\Phi \in L^{1}(D\times (0,T))$.\\ 
In order to show this assertion, first we show the uniform control property of both $f_n$ and $f$ and their gradients. Indeed, by the uniform boundedness of $(F_n)$,  applying the same estimates as in Lemma \ref{lem-priori} for both $g=f_n$ and $f$, respectively, we obtain
\begin{align*}
\begin{aligned}
&\|g\|_{L^{\infty}(0,T; L^p(D))}\le C\|f_0\|_{L^p(D)},\quad 1\le p\le \infty, \\
& \|\nabla_{\omega}g^{\frac{p}{2}}\|_{L^{2}(D\times (0,T))}\le C\|f_0\|_{L^p(D)}^{p/2},\quad 1\le p<\infty,
\end{aligned}
\end{align*}
where the positive constant $C$ only depends on $p$ and $T$. \\

Next, by H$\ddot{\mbox{o}}$lder's inequality follows that
\begin{align*}
\begin{aligned}
\int_0^T \int_{D}(f_n-f)^{p-1} \nabla_{\omega} f dxd\omega ds   &\le\Big(\int_{D} (f_n-f)^{p}dxd\omega\Big)^{1/2}
\Big(\int_{D} (f_n-f)^{p-2}|\nabla_{\omega}  f|^2 dxd\omega\Big)^{1/2}\\
&\le C\Big(\int_{D} (f_n^{p} + f^p) dxd\omega\Big)^{1/2}
\Big(\int_{D} | \nabla_{\omega} f^{\frac{p}{2}}|^2 dx d\omega\Big)^{1/2}\\
&\le C_0,
\end{aligned}
\end{align*}
and 
\begin{align*}
\begin{aligned}
\int_0^T \int_{D}(f_n-f)^{p-1} f dxd\omega ds   &\le\Big(\int_{D} (f_n-f)^{p}dxd\omega\Big)^{\frac{p-1}{p}}
\Big(\int_{D} f^{p} dxd\omega\Big)^{\frac{1}{p}}\\
&\le C\Big(\int_{D} (f_n^p+f^p)dxd\omega\Big)^{\frac{p-1}{p}}
\Big(\int_{D} f^{p} dxd\omega\Big)^{\frac{1}{p}}\\
&\le C_0,
\end{aligned}
\end{align*}
where the positive constant $C_0$ depends only on $\|f_0\|_{L^p(D)}$.\\

Therefore, the weight function $\Phi(x,w,t)$ can be estimated by
\begin{align*}
\begin{aligned}
\|\Phi \|_{L^{1}(D\times (0,T))}&\le C_*( \|(f_n-f)^{p-1} \nabla_{\omega} f\|_{L^{1}(D\times (0,T))}+\|(f_n-f)^{p-1}f\|_{L^{1}(D\times (0,T))})\\
&\le C_*C_0,
\end{aligned}
\end{align*}
where the positive constant  $C_*$  is given by
\begin{align*}
\begin{aligned}
C_*&= pd\Big[\Big((\| \nu^{\prime}(\omega\cdot F^*_n)\|_{L^{\infty}}+\| \nu^{\prime}(\omega\cdot F_n)\|_{L^{\infty}})+\| \nu^{\prime\prime}(\omega\cdot F^{**}_n)\|_{L^{\infty}}\| F\|_{L^{\infty}} \Big)\| F_n\|_{L^{\infty}}\\
&\quad+\| \nu(\omega\cdot F)\|_{L^{\infty}}+\| \nu^{\prime}(\omega\cdot F)\|_{L^{\infty}}\| F\|_{L^{\infty}} \Big],
\end{aligned}
\end{align*}
which does not depend on $n$ thanks to the uniform boundedness of the sequence $F_n$.\\

Hence, applying \eqref{weak-co} to \eqref{vanish}, it follows that
\begin{align}
\begin{aligned}\label{conclude}
&f_n\rightarrow f\quad\mbox{in} ~L^{p}(D\times (0,T)),\\
&\nabla_{\omega}f_n\rightarrow \nabla_{\omega}f\quad\mbox{in} ~L^{2}(D\times (0,T)). 
\end{aligned}
\end{align}

Finally, in order to complete the proof of Lemma~\ref{lem-compact},
 it remains to show  that \eqref{conclude} implies the strong convergence of the associated sequence $({J}_n)=(J(f_n))$ towards $J(f)$. 
Indeed,  Minkowski inequality, H$\ddot{\mbox{o}}$lder's inequality and Young's inequality yield
\begin{align}
\begin{aligned}\label{converge-J}
\|{J}_n-{J}\|_{L^{p}(U\times (0,T))}
&=\Big( \int_0^T \int_{U} \Big| \int_{\bbs^{d-1}} K*_{x}(f_n-f) \omega  d\omega\Big|^{p} dx ds \Big)^{\frac{1}{p}}\\
&\le \int_{\bbs^{d-1}}\Big(\int_0^T \int_{U} |K*_{x}(f_n-f)|^{p} dx ds \Big)^{\frac{1}{p}}d\omega\\
&\le C \Big(\int_0^T \int_{\bbs^{d-1}} \|K*_{x}(f_n-f)\|_{L^p(U)}^{p} d\omega ds \Big)^{\frac{1}{p}}\\
&\le C \|K\|_{L^1(\bbr^d)} \|{f}_n-f\|_{L^{p}(D\times (0,T))},
\end{aligned}
\end{align}
which completes the proof.
\end{proof}

\section{Proof of Existence - Theorem \ref{thm-exist}}
\setcounter{equation}{0}
%This section is devoted to the proof of Theorem \ref{thm-exist}. 
The proof of Theorem \ref{thm-exist} entices the construction of  an iteration scheme that  generates a sequence $(f_n)$, where $f_n$ is a solution to the linear equation \eqref{main-compact2}  at $n$-th step, 
with $F_n:=\Omega(f_{n-1})$ evaluated at the $(n-1)$-{th} solution $f_{n-1}$ obtained in the previous $(n-1)$-th step.

This first intuitive approach confronts  a difficulty since   such $n$-iteration scheme generating the sequence $f_n$ does not secure the non-zero momentum  $ |J(f_n)|> 0$, even if   $|J(f_{n-1})| >0$.\\  
In fact, if that would be the case, the term $\Omega(f_{n})$ would be undefined and therefore we could not secure it is bounded. 
In particular, since  the compactness properties of Lemma \ref{lem-priori}  and Lemma~\ref{lem-compact} require a bounded force term ( in \eqref{main-compact} and \eqref{main-compact2} respectively) then,  with  with at least the available tools developed in this manuscript, it would not be possible to secure an existence of a solution $f_{n+1}$  for the next $n$-iterative  step.

Hence, a way to avoid this difficulty can be acomplished by the  use of an $\eps$-regularization approach by adding an arbitrary $\eps>0$ parameter to the denominator of $\Omega(f_n)$, for all $n\in \bbn$. Such regularization generates a  double parameter $(\eps,n)$ sequence of  solutions $f_{\eps,n}$ 
that it is shown to satisfy the property  $|J(f_{\eps,n})|>0$ for  all $n\in \bbn$, uniformly in $\eps>0$.

\medskip
    
\subsection{The $\eps$-regularized equation}  The $\eps$-regularization approach consists in solving the non-linear problem \eqref{main} by adding $\eps>0$ to the denominator of $\Omega(f)$, 
\begin{align}
\begin{aligned} \label{main-app}
&\partial_t f_{\eps} + \omega\cdot\nabla_x f_{\eps} = -\nabla_{\omega}\cdot\Big(  f_{\eps}\nu(\omega\cdot \Omega_{\eps} )\bbp_{\omega^{\perp}}\Omega_{\eps} \Big) + \Delta_{\omega} f_{\eps},\\
&\Omega_{\eps}(f_{\eps})(x,t):= \frac{J_{\eps}(f_{\eps})(x,t)}{|J_{\eps}(f_{\eps})(x,t)|+{\eps}}, \quad J_{\eps}(f_{\eps})(x,t)= \int_{U \times\bbs^{d-1}} K(|x-y|)\omega f_{\eps} (y,\omega, t) dy d\omega\\
&f_{\eps}(x,\omega,0) = f_0(x,\omega),  \quad ~x\in U,~\omega\in\bbs^{d-1},~t>0.
\end{aligned}
\end{align}

This new non-linear $\eps$-problem is then solved by generating  a sequence of solutions $f_{\eps,n}$ to \eqref{main-compact2} with a bounded $F_{\eps,n}:= \Omega_{\eps}(f_{\eps, n-1})$ for the previous iterated solution $f_{\eps,n-1}$.   

In the sequel, we  show first that is possible to construct a sequence of solutions  $f_{\eps,n}$  converging to $f_{\eps}$ in $L^{p}(D\times (0,T))\cap L^{2}(U\times (0,T) ; H^1(\bbs^{d-1}))$, $1\leq p\leq \infty$ for any $\eps>0$, so the results remains true in the $\eps\to 0$ limit. 

The details of this procedure are as follows.  

\medskip

\subsection{Construction of approximate solutions}
The construction of  an $(\eps,n)$-sequence of approximate solutions $f_{\eps,n}$ to the non-linear $\eps$-regularized equation \eqref{main-app} is now done by the following iteration scheme.
For any fixed $\eps>0$,  set  $f_{\eps,0}(x,\omega,t):= f_0(x,\omega)$ to be the initial state associated to \eqref{main}.  Then,  define $f_{\eps,1}$ as the solution of the following linear initial value problem
\begin{align*}
\begin{aligned} 
&\partial_t f_{\eps,1}+ \omega\cdot\nabla_x f_{\eps,1} = -\nabla_{\omega}\cdot\Big( f_{\eps,1}\nu ( \omega \cdot {\Omega}_{\eps,0})\bbp_{\omega^{\perp}} {\Omega}_{\eps,0} \Big) + \Delta_{\omega} f_{\eps,1},\\
& {\Omega}_{\eps,0} (x,t) =\frac{J_{\eps,0}(x,t)}{|J_{\eps,0}(x,t)| +\eps},
 \quad J_{\eps,0}(x,t)= \int_{U \times\bbs^{d-1}} K(|x-y|)\omega f_{\eps,0} (y,\omega, t) dy d\omega\\
&f_{\eps,1}(x,\omega,0) = f_0(x,\omega).
\end{aligned}
\end{align*}
Inductively,  each $f_{\eps,n+1}$ is define to be the solution of the following linear initial value problem
\begin{align}
\begin{aligned} \label{main-linear}
&\partial_t f_{\eps,n+1}+ \omega\cdot\nabla_x f_{\eps,n+1} = -\nabla_{\omega}\cdot\Big( f_{\eps,n+1}\nu ( \omega \cdot {\Omega}_{\eps,n})\bbp_{\omega^{\perp}} {\Omega}_{\eps,n} \Big) + \Delta_{\omega} f_{\eps,n+1},\\
&{\Omega}_{\eps,n} (x,t) =\frac{J_{\eps,n}(x,t)}{|J_{\eps,n}(x,t)| +\eps},
 \quad J_{\eps,n}(x,t)= \int_{U \times\bbs^{d-1}} K(|x-y|)\omega f_{\eps,n} (y,\omega, t) dy d\omega\\
&f_{\eps,n+1}(x,\omega,0) = f_0(x,\omega).
\end{aligned}
\end{align}

The justification for  unique solvability of the $(\eps,n)$-approximate initial value problem \eqref{main-linear}, for $n\ge 1$, follows form the next lemma.

%For the notational simplicity, we may omit $\eps$-dependence in $f_{\eps, n}$ from now on.

\begin{lemma}\label{lem-linear}
For any $T>0$, $\eps>0$, $n\ge 1$, assume that $f_{\eps,n}$ is a given integrable function and $f_0$ satisfies \eqref{initial-con}. Then, there exists a unique solution $f_{\eps,n+1}\ge 0$ to the equation \eqref{main-linear} satisfying the $L^p$-estimates: for any $1\le p<\infty$,
\beq\label{p-est-2}
\|f_{\eps,n+1}\|_{L^{\infty}(0,T; L^p(D))} +\frac{2(p-1)}{p} \|\nabla_{\omega}f_{\eps,n+1}^{\frac{p}{2}}\|_{L^{2}(D\times (0,T))}^{\frac{2}{p}} \le e^{CT\frac{p}{p-1}} \|f_0\|_{L^p(D)},
\eeq
and
\beq\label{bound-est-2}
\|f_{\eps,n+1}\|_{L^{\infty}(D\times (0,T))} \le e^{CT}\|f_0\|_{L^{\infty}(D)}.
\eeq
\end{lemma} 
The proof of Lemma \ref{lem-linear} follows the same argument as Degond's proof in \cite{D}. We include its proof in Appendix for the reader's convenience. 

\medskip

\subsection{Passing to the limit as $n\rightarrow \infty$} 
The convergence of  $f_{\eps,n}$  towards some limit function $f_{\eps}$, which solves the regularized equation \eqref{main-app}, is secured by the following proposition.

\begin{proposition}\label{prop-app}
For a given $T>0$ and arbitrary $\eps>0$, if $f_0$ satisfies \eqref{initial-con}, then there exists a weak solution $f_{\eps} \ge 0$ to equation \eqref{main-app} satisfying the $L^p$-estimates: 
for $1\le p<\infty$,  
\beq\label{p-est-1}
\|f_{\eps}\|_{L^{\infty}(0,T; L^p(D))} +\frac{2(p-1)}{p} \|\nabla_{\omega}f_{\eps}^{\frac{p}{2}}\|_{L^{2}(D\times (0,T))}^{\frac{2}{p}} \le e^{CT\frac{p}{p-1}}\|f_0\|_{L^p(D)},
\eeq
and
\beq\label{bound-est-1}
\|f_{\eps}\|_{L^{\infty}(D\times (0,T))} \le e^{CT} \|f_0\|_{L^{\infty}(D)}.
\eeq
\end{proposition}
\begin{proof}
Since the sequence $({\Omega}_{\eps,n})$ defined in \eqref{main-linear} is bounded in $L^{\infty}(U\times (0,T))$, we use Lemma \ref{lem-compact} with $F_{\eps,n}=\Omega_{\eps,n}$. Thus, there exists a limit function $f_{\eps}$ such that, up to a subsequence, 
\begin{align*}
\begin{aligned}
&f_{\eps,n} \rightarrow f_{\eps} \quad\mbox{as}~n\rightarrow \infty ~\mbox{in}~L^{p}(D\times (0,T))\cap L^{2}(U\times (0,T) ; H^1(\bbs^{d-1})),\\
&J_{\eps,n} \rightarrow J_{\eps} \quad\mbox{as}~n\rightarrow \infty ~\mbox{in}~L^{p}(U\times (0,T)),
\end{aligned}
\end{align*}
that  yields 
\[
\Omega_{\eps,n} ~\rightarrow ~\Omega_{\eps}:=\frac{J_{\eps}}{|J_{\eps}| +\eps}\quad \mbox{as}~n \to \infty~\mbox{ in}~ L^{\infty}(0,T; L^p(D)).
\]
Indeed, 
\begin{align*}
\begin{aligned}
&\int_{U} |\Omega_{\eps,n}-\Omega_{\eps}|^p dx \\
&\qquad = \int_{U} \Big| \frac{\eps(J_{\eps,n}-J_{\eps}) + |J_{\eps}|(J_{\eps,n}-J_{\eps}) +J_{\eps}(|J_{\eps}|-|J_{\eps,n}|)}{(|J_{\eps,n}|+{\eps})(|J_{\eps}|+{\eps})} \Big|^p dx\\
&\qquad \le \frac{1}{\eps^{p}} \int_{U} \Big| \frac{\eps(J_{\eps,n}-J_{\eps}) + |J_{\eps}|(J_{\eps,n}-J_{\eps}) +J_{\eps}(|J_{\eps}|-|J_{\eps,n}|)}{|J_{\eps}|+{\eps}} \Big|^{p} dx\\
&\qquad\le C(\eps)\int_{U} \Big(|J_{\eps,n}-J_{\eps}|^{p} + |J_{\eps,n}-J_{\eps}|^{p} +||J_{\eps,n}|-|J_{\eps}||^{p} \Big) dx\\
&\qquad\le C(\eps)\int_{U} |J_{\eps,n}-J_{\eps}|^{p}dx.
\end{aligned}
\end{align*}

Therefore, the limit $f_{\eps}$ satisfies the following weak formulation of \eqref{main-app}: for all $\phi\in C^{\infty}_c (D\times [0,T))$,
 \begin{align*}
\begin{aligned}
&\int_0^t\int_{D}f_{\eps} \partial_t\phi +f_{\eps}\omega\cdot\nabla_x \phi  + f_{\eps}F_{\eps}\cdot\nabla_{\omega} \phi - \nabla_{\omega} f_{\eps} \cdot\nabla_{\omega}\phi dxd\omega ds\\
&\hspace{3cm} +\int_{D} f_0 \phi(0,\cdot) dxd\omega = 0,\\
&F_{\eps}= \nu(\omega\cdot \Omega_{\eps}) \bbp_{\omega^{\perp}}\Omega_{\eps} ,\quad \Omega_{\eps}(x,t)=\frac{ J_{\eps}(x,t)}{|J_{\eps}(x,t)| +\eps}.
\end{aligned}
\end{align*}
In addition, using Lemma \ref{lem-priori} together with the boundedness of $\Omega_{\eps}$ above, the  $L^p$  estimates from \eqref{p-est-1} and \eqref{bound-est-1} follow. 
\end{proof} 
 
 \medskip
%\subsection{Passing to the limit as $\eps\rightarrow 0$} 

%by showing the convergence of \eqref{main-app} to \eqref{main} in the weak sense, as $\eps\rightarrow 0$. For the convenience, we denote a sequence $f_n:=f_{\eps_n}$ for a convergent sequence $\eps_n\rightarrow 0$, then consider a sequence
\subsection{Passing to the limit as $\eps\rightarrow 0$} 
The proof of Theorem \ref{thm-exist} is completed after showing the convergence from \eqref{main-app} to \eqref{main}  as $0<\eps\rightarrow 0$, in the weak sense. 
In fact, it is enough to show such limit for  any convergent sequence $0<\eps_k\rightarrow 0$.

First,  consider a sequence
\[
F_k:=\frac{J_{\eps_k}}{|J_{\eps_k}| +\eps_k},\quad J_{\eps_k}=\int_{\bbs^{d-1}} \omega f_{\eps_k} d\omega.
\]
Since such sequence is bounded in $L^{\infty}(U\times (0,T))$ uniformly in $\eps_k$, Lemma \ref{lem-compact} can be applied, so 
 there exists a limit function $f$ such that, up to a subsequence, 
\begin{align}
\begin{aligned}\label{J-con}
&f_{\eps_k}\rightarrow f \quad\mbox{as}~k\rightarrow \infty ~\mbox{in}~L^{p}(D\times (0,T))\cap L^{2}(U\times (0,T); H^1(\bbs^{d-1})),\\
&J_{\eps_k} \rightarrow J \quad\mbox{as}~k\rightarrow \infty ~\mbox{in}~L^{p}(U\times (0,T)).
\end{aligned}
\end{align}

Next, in order to see that  $f$ is the weak solution to \eqref{main} it is enough to show that $f$ satisfies the weak formulation \eqref{weak-form} as a limit of the following formulation for \eqref{main-app}: 
\begin{align*}
\begin{aligned}
&\int_0^t\int_{D}f_{\eps_k} \partial_t\phi +f_{\eps_k}\omega\cdot\nabla_x \phi  + f_{\eps_k} \nu \Big(\frac{\omega\cdot J_{\eps_k}}{|J_{\eps_k}|+{\eps_k}}\Big) \bbp_{\omega^{\perp}}\frac{J_{\eps_k}}{|J_{\eps_k}| +\eps_k} \cdot\nabla_{\omega} \phi -  \nabla_{\omega} f_{\eps_k} \cdot\nabla_{\omega}\phi dxd\omega ds\\
&\hspace{3cm} +\int_{D} f_0 \phi(0,\cdot) dxd\omega = 0,
\end{aligned}
\end{align*}
for any $\phi\in C^{\infty}_c (D\times [0,T))$.

By the convergence of $f_{\eps_k}$ in \eqref{J-con}, clearly all linear terms in the above formulation converge to their corresponding terms in \eqref{weak-form}. On the other hand, the convergence of the nonlinear term requires further justification provided in the following Lemma.

\begin{lemma}\label{lemma4.2}
Assume $|J(x,t)|>0$ as in \eqref{assume}. Then, as $k\to\infty$,
\begin{multline}\label{eps1}
\int_0^t\int_{D}  f_{\eps_k} \nu \Big(\frac{\omega\cdot J_{\eps_k}}{|J_{\eps_k}|+{\eps_k}}\Big) \bbp_{\omega^{\perp}}\frac{J_{\eps_k}}{|J_{\eps_k}| +\eps_k} \cdot\nabla_{\omega} \phi dxd\omega ds \ \ \ 
{\longrightarrow}\\
  \int_0^t\int_{D} f \nu \Big(\frac{\omega \cdot J}{|J|}\Big) 
\bbp_{\omega^{\perp}}  \frac{J}{|J|} \cdot\nabla_{\omega} \phi dxd\omega ds.
\end{multline}
\end{lemma}

\begin{proof}
By the properties  \eqref{formula-1} of calculus on the sphere applied to the  projection operator $\bbp_{\omega^{\perp}}$ is   the identity operator acting on gradient functions of the sphere $\bbs^{d-1}$, that is \break $\bbp_{\omega^{\perp}}\cdot\nabla_{\omega}\Phi = \nabla_{\omega}\Phi$ holds for any 
test function $\Phi$ of $w\in  \bbs^{d-1}$.
Then,  the limit as $k \to \infty$ in \eqref{eps1} is identical to show   the analog limit for the formulation without the   projection operator. That is, for $k \to \infty$
\begin{equation}\label{eps2}
\int_0^t\int_{D}  f_{\eps_k} \nu \Big(\frac{\omega\cdot J_{\eps_k}}{|J_{\eps_k}|+{\eps_k}}\Big)\frac{J_{\eps_k}}{|J_{\eps_k}| +\eps_k} \cdot\nabla_{\omega} \phi dxd\omega ds
 \rightarrow \int_0^t\int_{D} f \nu \Big(\frac{\omega \cdot J}{|J|}\Big) 
 \frac{J}{|J|} \cdot\nabla_{\omega} \phi dxd\omega ds\, .
\end{equation}

We first control the integrand in \eqref{eps2} using the estimates \eqref{bound-est-1} and boundedness of $\nu$, so that there is a uniform constant $C$ such that
\[
\Big\|f_{\eps_k} \nu \Big(\frac{\omega\cdot J_{\eps_k}}{|J_{\eps_k}|+{\eps_k}}\Big) \frac{J_{\eps_k}}{|J_{\eps_k}| +\eps_k}\Big\|_{L^{\infty}(D\times (0,T))} \le
\|f_{\eps_k}\|_{L^{\infty}(D\times (0,T))} \|\nu\|_{L^{\infty}} \le C,
\]
which implies,  for some $F$, that 
\[
f_{\eps_k} \nu \Big(\frac{\omega\cdot J_{\eps_k}}{|J_{\eps_k}|+{\eps_k}}\Big) \frac{J_{\eps_k}}{|J_{\eps_k}| +\eps_k}\rightharpoonup F\quad \mbox{weakly}-*~\mbox{in}~L^{\infty}(D\times (0,T)). 
\]

Then, it remains to show that
\[
F = f \nu \Big(\frac{\omega\cdot J}{|J|}\Big) \frac{J}{|J|},\quad \mbox{on} ~\{ (t, x,\omega) \in (0,T]\times U\times \bbs^{d-1}~|~ |J(x,t)| >0 \}.
\]
In order to obtain this last identity,  we consider the bounded set
\[
X_{R,\delta} : = \{ (t, x,\omega) \in (0,T]\times B_{R}(0) \times \bbs^{d-1}~|~ |J(x,t)| >\delta \},
\]
where $R$ and $\delta$ are any positive constants, and $B_R(0)$ denote the ball in $U$, with radius $R$, centered at $0$.\\ 

Since $f_{\eps_k} \rightarrow f$ and $J_{\eps_k} \rightarrow J$ a.e. on $X_{R,\delta}$ by \eqref{J-con}, then by   Egorov's theorem, for any $\eta>0$, there exists a $Y_{\eta}\subset X_{R,\delta}$ such that $|X_{R,\delta}\backslash Y_{\eta} | <\eta$ and
\[
f_{\eps_k} \rightarrow f,\quad J_{\eps_k} \rightarrow J \quad\mbox{in}~L^{\infty} (Y_{\eta})\, ,
\] 
and so, for sufficiently large $k$,
\[
|J_{\eps_k}(x,t)| > \frac{\delta}{2}\quad \mbox{for} ~(x,t)\in Y_{\eta}.
\]

Therefore, the $L^\infty(Y_\eta)$ $\eps$-convergence follows from 
\begin{align*}
\begin{aligned}
\Big\|\frac{J_{\eps_k}}{|J_{\eps_k}| +\eps_k} - \frac{J}{|J|} \Big\|_{L^{\infty} (Y_{\eta})} 
&=  \Big\| \frac{|J|(J_{\eps_k}-J) +J(|J|-|J_{\eps_k}|)-\eps_k J }{(|J_{\eps_k}|+{\eps_k})|J| } \Big\|_{L^{\infty} (Y_{\eta})}\\
&\le \frac{2}{\delta}  \Big\| \frac{ |J|(J_{\eps_k}-J) +J(|J|-|J_{\eps_k}|)-\eps_k J}{|J|} \Big\|_{L^{\infty} (Y_{\eta})}\\
&\le \frac{2}{\delta}\Big( \|J_{\eps_k}-J\|_{L^{\infty} (Y_{\eta})} +\||J_{\eps_k}|-|J|\|_{L^{\infty} (Y_{\eta})} -\eps_k \Big) \quad \rightarrow~ 0,
\end{aligned}
\end{align*}
that yiels
\begin{align*}
\begin{aligned}
&\Big\| f_{\eps_k} \nu \Big(\frac{\omega\cdot J_{\eps_k}}{|J_{\eps_k}|+{\eps_k}}\Big) \frac{J_{\eps_k}}{|J_{\eps_k}| +\eps_k}
-f \nu \Big(\frac{\omega\cdot J}{|J|}\Big)\frac{J}{|J|} \Big\|_{L^{\infty} (Y_{\eta})} \\
&\quad = \Big\| f_{\eps_k} \Big[ \nu \Big(\frac{\omega\cdot J_{\eps_k}}{|J_{\eps_k}|+{\eps_k}}\Big)
-\nu \Big(\frac{\omega\cdot J}{|J|}\Big) \Big] \frac{J_{\eps_k}}{|J_{\eps_k}| +\eps_k} \Big\|_{L^{\infty} (Y_{\eta})}\\
&\qquad + \Big\| f_{\eps_k} \nu \Big(\frac{\omega\cdot J}{|J|}\Big)
\Big(\frac{J_{\eps_k}}{|J_{\eps_k}| +\eps_k}-\frac{J}{|J|}\Big) \Big\|_{L^{\infty} (Y_{\eta})} + \Big\|(f_{\eps_k}-f) \nu \Big(\frac{\omega\cdot J_{\eps_k}}{|J_{\eps_k}|+{\eps_k}}\Big) \frac{J_{\eps_k}}{|J_{\eps_k}| +\eps_k}\Big\|_{L^{\infty} (Y_{\eta})}\\
&\quad\le C\|f_{\eps_k}\|_{L^{\infty}} ( \|\nu^{\prime}\|_{L^{\infty}}+ \|\nu\|_{L^{\infty}}   )   
\Big\| \frac{J_{\eps_k}}{|J_{\eps_k}|+{\eps_k}}- \frac{J}{|J|} \Big\|_{L^{\infty}(Y_{\eta})} +C\|f_{\eps_k}-f\|_{L^{\infty}(Y_{\eta})} \|\nu\|_{L^{\infty} } \rightarrow 0.
\end{aligned}
\end{align*}

Hence, the following identity holds 
\[
F = f \nu \Big(\frac{\omega\cdot J}{|J|}\Big) \frac{J}{|J|} \quad \mbox{on}~Y_{\eta}\, ,
\]
and,  since $\eta$, $R$ and $\delta$ are arbitrary, taking $\eta,\delta\to 0$ and $R\to \infty$, it follows
\[
F = f \nu \Big(\frac{\omega\cdot J}{|J|}\Big) \frac{J}{|J|} \quad \mbox{on}~\{ (t, x,\omega) \in (0,T]\times U \times \bbs^{d-1}~|~ |J(x,t)| >0 \}\, .
\]
which completes the proof of Lemma~\ref{lemma4.2}.
\end{proof}

Finally, thanks to  lemma~\ref{lemma4.2} and \eqref{J-con}, it follows that $f$ satisfies the weak formulation \eqref{weak-form}.
In addition,  estimates \eqref{p-est-0} and \eqref{bound-est-0} follow directly from  \eqref{p-est-1} and \eqref{bound-est-1}, respectively.
Therefore, the proof of Theorem~\ref{thm-exist} is now completed.

\medskip

\section{Proof of Uniqueness - Theorem \ref{thm-unique}}
The uniqueness argument is considered in the subclass $\mathcal{A}_{\alpha}$ of weak solutions constructed in Theorem \ref{thm-exist}. 
Let $f$ and $g$ be any weak solutions to the initial value problem \eqref{main} in $\mathcal{A}_{\alpha}$. A straightforward computation yields that 
\begin{align}\label{uni1}
\begin{aligned}
&\frac{1}{2}\frac{d}{dt}\int_{\bbt^d\times\bbs^{d-1}} |f-g|^2 dxd\omega  +\int_{\bbt^d\times\bbs^{d-1}} | \nabla_{\omega} (f-g)|^2 dx d\omega \\
&\quad=-  \int_{\bbt^d\times\bbs^{d-1}} (f-g) \nabla_{\omega}\cdot\Big( (f-g)\nu(\omega\cdot \Omega(f))\bbp_{\omega^{\perp}} \Omega(f)  \Big)dxd\omega \\
&\qquad - \int_{\bbt^d\times\bbs^{d-1}} (f-g)\nabla_{\omega}\cdot\Big( g (\nu(\omega\cdot \Omega(f))-\nu(\omega\cdot \Omega(g))) \bbp_{\omega^{\perp}} \Omega(f) \Big)dxd\omega \\
&\qquad - \int_{\bbt^d\times\bbs^{d-1}} (f-g)\nabla_{\omega}\cdot\Big( g \nu(\omega\cdot \Omega(g))\bbp_{\omega^{\perp}} (\Omega(f)- \Omega(g) ) \Big)dxd\omega\\
&\quad=:{J}_1 + {J}_2+{J}_3.
\end{aligned}
\end{align}
Using the same estimates applied to  $\mathcal{J}_1$ in the proof of Lemma~\ref{lem-compact}, we can also estimate
\begin{align*}
\begin{aligned}
{J}_1 &\le \frac{1}{4} \int_{\bbt^d\times\bbs^{d-1}} | \nabla_{\omega} (f-g)|^2 dx d\omega + C\int_{\bbt^d\times\bbs^{d-1}} |f-g|^2 dxd\omega.
\end{aligned}
\end{align*}
Next,  $J_2$ and $J_3$ can also be estimated same approach  from  \eqref{converge-J}  in Lemma~\ref{lem-compact}, to get
\[
\|{J}(f)-{J}(g)\|_{L^{2}(\bbt^d)}\le C \|K\|_{L^1(\bbt^d)} \|f-g\|_{L^{2}(\bbt^d\times\bbs^{d-1})}.
\]
Moreover, since $|J(f)|\ge\alpha$ in the set $\mathcal{A}_{\alpha}$, then  the difference of alignment forces for any two weak solutions is  controlled by
\begin{align*}
\begin{aligned}
|\Omega(f)-\Omega(g)| &\le \frac{\Big| |J(g)|(J(f)- J(g))- J(g)(|J(f)|-|J(g)|)\Big|}{\alpha |J(g)|}\\
& \le \frac{2}{\alpha}|J(f)-J(g)| \, .
\end{aligned}
\end{align*}
that yields
\[
\|\Omega(f)-\Omega(g)\|_{L^{2}(\bbt^d)}\le C \|f-g\|_{L^{2}(\bbt^d\times\bbs^{d-1})}.
\]
Therefore,  by property \eqref{regularity} for any weak solution, the control of term  $J_2$  in \eqref{uni1}  follows from
\begin{align*}
\begin{aligned}
J_2&= \int_{\bbt^d\times\bbs^{d-1}} \nabla_{\omega}(f-g)\cdot \bbp_{\omega^{\perp}} \Omega(f) g (\nu(\omega\cdot \Omega(f))-\nu(\omega\cdot \Omega(g))) dxd\omega \\
&\le \|g\|_{L^{\infty}}\|\nu^{\prime}\|_{L^{\infty}}\int_{\bbt^d\times\bbs^{d-1}}| \nabla_{\omega}(f-g)|  |\Omega(f)-\Omega(g)| dxd\omega\\
&\le \frac{1}{4} \int_{\bbt^d\times\bbs^{d-1}} | \nabla_{\omega} (f-g)|^2 dx d\omega + C\int_{\bbt^d\times\bbs^{d-1}} |f-g|^2 dxd\omega.
\end{aligned}
\end{align*}
Likewise, the control of the last term  $J_3$  in \eqref{uni1}  follows, since 
\[
J_3\le \frac{1}{4} \int_{\bbt^d\times\bbs^{d-1}} | \nabla_{\omega} (f-g)|^2 dx d\omega + C\int_{\bbt^d\times\bbs^{d-1}} |f-g|^2 dxd\omega.
\]

Hence, gathering the above estimates and using Gronwall's inequality, we have
\[
\int_{\bbt^d\times\bbs^{d-1}} |f-g|^2 dxd\omega \le e^{CT}\int_{\bbt^d\times\bbs^{d-1}} |f_0-g_0|^2 dxd\omega,
\] 
which implies the uniqueness of weak solutions to the initial value problem \eqref{main}  in $\mathcal{A}_{\alpha}$.

\bigskip

\section{Conclusion}
We have shown the existence of global weak solutions to problem \eqref{main} (as well for $J$ defined as in  \eqref{main-0}) in a subclass of solutions with the non-zero local momentum. These solutions are unique on the subclass of solutions in the $d$ dimensional
torus whose mean speed is uniformly bounded below by a strictly positive constant.
 
An important future work would be to remove our assumption on the non-zero local momentum. The main difficulty is due to the lack of momentum conservation for solutions to problem \eqref{main}, and canonical entropy associated to the  equation in \eqref{main}. 
Thus, at this point,  we have neither suitable functional spaces nor distances to study the behavior of solutions whose momentum may vanish locally. This difficulty is related to the issue on stability of solutions to \eqref{main}. 
Another future work is to extend the uniqueness result to the whole spatial domain $\bbr^d$.

\begin{appendix}
\setcounter{equation}{0}
\section{Proof of Lemma \ref{lem-linear}} 
For the notational simplicity, we omit the subindex $n+1$ in \eqref{main-linear}. Our goal is to prove  existence of solutions $f$ to the linear equation
\begin{align}
\begin{aligned} \label{li}
&\partial_t f+ \omega\cdot\nabla_x f = -\nabla_{\omega}\cdot\Big( f \nu ( \omega\cdot \bar{\Omega})\bbp_{\omega^{\perp}}    \bar{\Omega} \Big) + \Delta_{\omega} f,\\
&\bar{\Omega}=\frac{\bar{J}(x,t)}{|\bar{J}(x,t)|+{\eps}},\quad \bar{J}(x,t)= \int_{D} K(|x-y|)\omega g (y,\omega, t) dy d\omega,\\
&f(x,\omega,0) = f_0(x,\omega),
\end{aligned}
\end{align}
where $g$ is just a given integrable function.\\

We begin by rewriting \eqref{li} as
\begin{align}
\begin{aligned} \label{linear-re}
&\partial_t f+ \omega\cdot\nabla_x f + \nu (\omega \cdot \bar{\Omega})\bbp_{\omega^{\perp}} \bar{\Omega} \cdot \nabla_{\omega}f\\
&\qquad + f \nu^{\prime} (\omega \cdot \bar{\Omega}) |\bbp_{\omega^{\perp}} \bar{\Omega}|^2   
-(d-1) f \nu(\omega \cdot \bar{\Omega})\omega\cdot \bar{\Omega}  - \Delta_{\omega} f =0,\\
&f(x,\omega,0) = f_0(x,\omega),
\end{aligned}
\end{align}
where formulas \eqref{formula-2} on projections and calculus on the sphere were used.\\

Next, taking  $\bar{f}(x,\omega,t) := e^{-\lambda t} f(x,\omega, t)$ for a given $\lambda>0$,  It leads to the modified initial value problem
\begin{align}
\begin{aligned} \label{linear-bar}
&\partial_t \bar{f}+ \omega\cdot\nabla_x \bar{f} +\psi_1\cdot \nabla_{\omega}\bar{f} +\Big(\lambda + \psi_2 +\psi_3 \Big) \bar{f} - \Delta_{\omega} \bar{f} =0,\\
&\bar{f}(x,\omega,0) = f_0(x,\omega),
\end{aligned}
\end{align}
where the functions $\psi_1$, $\psi_2$ and $\psi_3$ are given by
\begin{align*}
\begin{aligned} 
&\psi_1(x,\omega,t)=\nu (\omega \cdot \bar{\Omega}) \bbp_{\omega^{\perp}} \bar{\Omega},\\
& \psi_2(x,\omega,t)=\nu^{\prime} (\omega \cdot \bar{\Omega}) |\bbp_{\omega^{\perp}} \bar{\Omega}|^2,\\
&\psi_3(x,\omega,t)=-(d-1) \nu(\omega \cdot \bar{\Omega})\omega\cdot \bar{\Omega}\, , 
\end{aligned}
\end{align*}
respectively. Now, since $|\bar{\Omega}|\le 1$ and the smooth function $\nu$ is bounded, then $\psi_1$, $\psi_2$ and $\psi_3$ are also bounded.
Therefore, by  J. L. Lions' existence theorem in \cite{L},  the existence of a solution for \eqref{linear-bar} follows from the same argument given  by Degond 
 in \cite{D}. That means,  equation \eqref{linear-bar} has a solution $\bar{f}$ in the space
\[
Y:=\{ f\in L^{2}([0,T]\times U; H^1(\bbs^{d-1}))~|~ 
\partial_t f + \omega\cdot\nabla_x f \in L^{2}([0,T]\times U; H^{-1}(\bbs^{d-1}))   \}.
\]
Furthermore, by the Green's formula used  in  \cite{D}, then  the following identity holds,  for any $f \in Y$,
\begin{align}
\begin{aligned} \label{green}
\langle\partial_t f + \omega\cdot\nabla_x f, f\rangle = \frac{1}{2}\int_{D}( |f(x,\omega,T)|^2 - |f(x,\omega,0)|^2 )dxd\omega,
\end{aligned}
\end{align}
where $\langle\cdot,\cdot\rangle$ denotes the pairing of $L^{2}([0,T]\times U; H^{-1}(\bbs^{d-1}))$ and $L^{2}([0,T]\times U; H^{1}(\bbs^{d-1}))$.

This identity  \eqref{green} is needed below to show  uniqueness of solutions $f$ in $Y$ as follows.\\

Let  $\bar{f}\in Y$  be a solution to \eqref{linear-bar} with initial data $f_0=0$. Then,  by   \eqref{green},  it follows
\begin{align}
\begin{aligned}
0&=\langle\partial_t \bar{f}+ \omega\cdot\nabla_x \bar{f} +\psi_1\cdot \nabla_{\omega}\bar{f} +(\lambda + \psi_2 +\psi_3 ) \bar{f} - \Delta_{\omega} \bar{f}, \bar{f}\rangle\\
&=\frac{1}{2}\int_{D}|\bar{f}(x,\omega,T)|^2 dxd\omega -\frac{1}{2}\int_{D} \nabla_{\omega}\cdot\psi_1 |\bar{f}|^2 dxd\omega\\
&\qquad + \int_{D} (\lambda + \psi_2 +\psi_3 )  |\bar{f}|^2 dxd\omega 
+ \int_{D} |\nabla_{\omega}\bar{f}|^2 dxd\omega   \label{unique-1} \\
&\ge \Big(\lambda - \frac{1}{2}\|\nabla_{\omega}\cdot\psi_1\|_{L^{\infty} ([0,T]\times D)}
-\| \psi_2\|_{L^{\infty} ([0,T]\times D)}-\|\psi_3\|_{L^{\infty} ([0,T]\times D)}  \Big) \int_{D} |\bar{f}|^2 dxd\omega. 
\end{aligned}
\end{align} 
Next, since
\begin{align*}
\begin{aligned} 
 \nabla_{\omega}\cdot\psi_1 &= \nu^{\prime} (\omega \cdot \bar{\Omega}) \nabla_{\omega} (\omega \cdot \bar{\Omega}) \cdot \bbp_{\omega^{\perp}} \bar{\Omega} + \nu(\omega \cdot \bar{\Omega}) \nabla_{\omega} \cdot \bbp_{\omega^{\perp}} \bar{\Omega} \\
 &=\nu^{\prime} (\omega \cdot \bar{\Omega}) |\bbp_{\omega^{\perp}} \bar{\Omega}|^2 -(d-1) \nu(\omega \cdot \bar{\Omega})\omega\cdot \bar{\Omega},
\end{aligned}
\end{align*}
the term $ \nabla_{\omega}\cdot\psi_1$ is bounded.  Thus,  choosing $\lambda$ such that 
\beq\label{lambda}
\lambda > \frac{1}{2}\|\nabla_{\omega}\cdot\psi_1\|_{L^{\infty} ([0,T]\times D)}
+\| \psi_2\|_{L^{\infty} ([0,T]\times D)}+\|\psi_3\|_{L^{\infty} ([0,T]\times D)}, 
\eeq
then, estimate \eqref{unique-1} yields $\bar{f}=0$, which proves the uniqueness of the linear equation \eqref{linear-bar}. 
Therefore, \eqref{linear-bar} has a unique solution $\bar{f}\in L^{2}([0,T]\times U; H^1(\bbs^{d-1}))$.\\ 

Furthermore, since $f_0\ge 0$ and $f_0\in L^{\infty}(D)$, by  a similar argument as in \eqref{unique-1}, 
\[
 \bar{f} \ge 0\quad \mbox{and}\quad \bar{f} \in L^{\infty} ([0,T]\times D) .
\]
Indeed, using  the following identity from \cite{D}  on any $f\in Y$, with ${f}_{-}:=\max(-{f} ,0)$,
\[
\langle\partial_t f + \omega\cdot\nabla_x f, {f}_{-}\rangle = \frac{1}{2}\int_{D}(|{f}_{-}(x,\omega,0)|^2- |{f}_{-}(x,\omega,T)|^2  )dxd\omega \, .
\]

Then, since ${f}_{-}(x,\omega,0)=0$ when $f_0\ge 0$, it follows
\begin{align*}
\begin{aligned}
0&=\langle\partial_t \bar{f}+ \omega\cdot\nabla_x \bar{f} +\psi_1\cdot \nabla_{\omega}\bar{f} +(\lambda + \psi_2 +\psi_3 ) \bar{f} - \Delta_{\omega} \bar{f}, \bar{f}_-\rangle\\
&=-\frac{1}{2}\int_{D}|\bar{f}_-(x,\omega,T)|^2 dxd\omega +\frac{1}{2}\int_{D} \nabla_{\omega}\cdot\psi_1 |\bar{f}_-|^2 dxd\omega\\
&\qquad - \int_{D} (\lambda + \psi_2 +\psi_3 )  |\bar{f}_-|^2 dxd\omega 
- \int_{D} |\nabla_{\omega}\bar{f}_-|^2 dxd\omega \\
&\le -\Big(\lambda - \frac{1}{2}\|\nabla_{\omega}\cdot\psi_1\|_{L^{\infty} ([0,T]\times D)}
-\| \psi_2\|_{L^{\infty} ([0,T]\times D)}-\|\psi_3\|_{L^{\infty} ([0,T]\times D)}  \Big) \int_{D} |\bar{f}_-|^2 dxd\omega. 
\end{aligned}
\end{align*}  

Using the same $\lambda$ as in \eqref{lambda}, yields $\bar{f}_-=0$, which proves $\bar{f}\ge 0$.\\
The same argument also deduces that
\[
\|\bar{f}\|_{L^{\infty} ([0,T]\times D)} \le \|f_0\|_{L^{\infty} (D)}.
\]

Finally, using the transformation $f(x,\omega, t) = e^{\lambda t}\bar{f}(x,\omega,t)$, the results hold for solutions of   \eqref{linear-re} as well. 
In addition, since the $\bar{f}$ properties are invariant under such transformation, then the proof of existence is completed, and  estimates \eqref{p-est-2} and \eqref{bound-est-2} follow directly from Lemma \ref{lem-priori} together with boundedness of $\bar{\Omega}$. 
\end{appendix}

\end{document}